\documentclass[12pt]{article}

\usepackage{setspace}
\usepackage{amsmath,amscd}
\usepackage{amsfonts}
\usepackage{verbatim}
\usepackage{amsthm, thmtools}
\usepackage{mathtools}
\usepackage[english]{babel}
\usepackage{amssymb}
\usepackage[T1]{fontenc}
\usepackage{enumerate}
\usepackage{lmodern}
\usepackage[shortlabels]{enumitem}
\usepackage[a4paper,bindingoffset=0.2in,left=1in,right=1in,top=1in,bottom=1in,footskip=.6in]{geometry}
\usepackage{blindtext}
\usepackage{leftidx}
\usepackage{mathrsfs}
\usepackage[linesnumbered,ruled]{algorithm2e}
\usepackage{setspace}
\usepackage{tikz}

\usepackage{textcomp}
\usetikzlibrary{positioning}
\usetikzlibrary{matrix}
\usetikzlibrary{arrows}
\usetikzlibrary{calc}
\usepackage{anyfontsize}

\usepackage{enumerate}
\usepackage{fullpage}
\usepackage{titlesec}
\titleformat*{\section}{\large\bfseries}
\usepackage[autostyle]{csquotes} 
\usepackage{bbm}
\usepackage{bm}

\usepackage{xypic}
\input xy

\numberwithin{equation}{section}

\newtheorem{theorem}{Theorem}[section]
\newtheorem*{theorem*}{Theorem}

\theoremstyle{definition}
\newtheorem{definition}[theorem]{Definition}
\newtheorem{Ex}[theorem]{Example}
\newtheorem{Rem}[theorem]{Remark}

  \makeatletter
\renewcommand*{\@cite@ofmt}{\bfseries\hbox}
\makeatother

\newcommand{\oi}{\omega^i}

\newcommand{\owt}{\widetilde{\omega}}

\renewcommand{\o}{\omega}

\newcommand{\Jinf}{J^{\infty}}

\newcommand{\Ht}{\widetilde{\H}}

\renewcommand{\Xi}{X^i}

\newcommand{\uaJ}{u^\alpha_J}

\newcommand{\Uh}{\widehat{U}}

\newcommand{\UhaJ}{\widehat{U}^\alpha_J}

\newcommand{\ua}{u^\alpha}

\newcommand{\UhaJi}{\widehat{U}^\alpha_{J, i}}

\newcommand{\resone}{\restrictbig{\one}}

\xyoption{all}

\newcommand{\Sb}{\bar{S}}

\renewcommand{\th}{^{\text{th}}}

\newcommand{\p}{\varphi}

\newcommand{\Nz}{\mathbb{N}_0}

\newcommand{\B}{\mathcal{B}}

\newcommand{\Bt}{\widetilde{\mathcal{B}}}

\newcommand{\E}{\mathcal{E}}
\newcommand{\D}{\mathcal{D}}
\newcommand{\X}{\mathcal{X}}
\newcommand{\U}{\mathcal{U}}
\newcommand{\G}{\mathcal{G}}
\renewcommand{\S}{\mathcal{S}}
\newcommand{\V}{\mathcal{V}}

\renewcommand{\H}{\mathcal{H}}

\newcommand{\C}{\mathcal{C}}

\newcommand{\R}{\mathbb{R}}

\newcommand{\dis}{\displaystyle}

\newcommand{\resx}{\restrict{x}}

\newcommand{\resz}{\restrict{z}}

\newcommand\one{\mathbbm{1}}

\newcommand{\bpm}{\begin{pmatrix}}
\newcommand{\epm}{\end{pmatrix}}
\newcommand{\bbm}{\begin{bmatrix}}
\newcommand{\ebm}{\end{bmatrix}}
\newcommand{\beq}{\begin{equation}}
\newcommand{\eeq}{\end{equation}}
\newcommand{\bex}{\begin{Ex}}
\newcommand{\eex}{\end{Ex}}
\newcommand{\bre}{\begin{Rem}}

\newcommand\av{^\alpha}

\def\restrict#1{\raise-.3ex\hbox{\scriptsize\ensuremath|}_{#1}}
\def\restrictbig#1{\raise-.3ex\hbox{\large\ensuremath|}_{#1}}
\def\irestrict#1{\raise+.7ex\hbox{\scriptsize\ensuremath|}^{#1}}

\def\comp{\raise 1pt \hbox{$\,\scriptstyle\circ\,$}}

\newcommand{\pder}[2][]{\frac{\partial#1}{\partial#2}}

\begin{document}
\begin{center}
\begin{Large}
Involutive Moving Frames II; The Lie-Tresse theorem\end{Large}
\vskip 1cm
\begin{tabular*}{1.0\textwidth}{@{\extracolsep{\fill}} ll}
\"Orn Arnaldsson \\
Department of Mathematics\\
University of Iceland\\
Reykjavík, Ssn.\ 600169-2039 \\
{\tt ornarnalds@hi\!.\!is} \\
\end{tabular*}
\end{center}

\vskip 0.25cm\noindent
{\bf Keywords}:   moving frames, equivalence problems,  differential invariants
\vskip 0.25cm\noindent
{\bf 2020 Mathematics subject classification}:  22F05, 58A15, 53A55

\vskip 0.5cm

 
\begin{abstract}
This paper continues the project, begun in \cite{IMF}, of harmonizing Cartan's classical equivalence method and the modern equivariant moving frame in a framework dubbed \emph{involutive moving frames}. As an attestation of the fruitfulness of our framework, we obtain a new, constructive and intuitive proof of the Lie-Tresse theorem (Fundamental basis theorem) and a first general upper bound on the minimal number of generating differential invariants for Lie pseudo-groups. Further, we demonstrate the computational advantages of this framework by studying the equivalence problem for first order PDE in two independent variables and one dependent variable under point transformations. 
\end{abstract}

\section{Introduction}
This paper continues the project, begun in \cite{IMF}, of relating Cartan's equivalence method to the equivariant moving frame for pseudo-groups. The major step was accomplished in \cite[Theorem 3.12]{IMF} where Cartan's classical solution to the congruence problem for submanifolds of finite dimensional Lie groups was extended to the groupoids associated to Lie pseudo-groups. Cartan's solution in the finite dimensional case underlies his classical method of moving frames as well as the modern equivariant moving frame, \cite{FO-1999} and its infinite dimensional analog harmonized the finite and infinite dimensional methods of equivalence. However, the key theorem \cite[Theorem 3.12]{IMF} was most directly applicable to so-called horizontal actions which left out some important equivalence problems, particularly those for partial differential equations. The current paper completes this picture by clarifying precisely which equivalence problems can be solved using some version of \cite[Theorem 3.12]{IMF} and how the theorem applies.

There are two major consequences of this, both of which indicate that our framework is fruitful for equivalence problems for Lie pseudo-groups. First, we obtain a new, constructive and intuitive proof of the Lie-Tresse theorem on the finiteness of generators of the differential invariant algebra of Lie pseudo-groups. Our proof is much simpler than previous modern proofs, \cite{KL-2006, KL-2016, OP-2009}, and, interestingly, makes critical use of the finite dimensional theory. Second, there are enormous computational advantages to the combination of the equivariant moving frame and Cartan's geometry. We study, in Section \ref{sec:1stpde}, the general equivalence problem of first order PDE
\[
u_y=f(x, y, u, u_x)
\]  
under point transformations and characterize those equations that are point-equivalent to $u_y=0$ and $u_y=u_x^2$. The branching found here are new. A forthcoming paper is dedicated to the computational advantages of our framework.

The paper's setup is as follows. In Section \ref{sec:LG} we give a quick review of Lie pseudo-groups and equivariant moving frames and refine the definition of a partial moving frame from \cite{IMF}. We explain how the key results from \cite{IMF} carry over to the class of ``vertically integrable'' pseudo groups, leading to the equivalence method in Theorem \ref{thm:total}. In the following section we nail down the precise condition on pseudo-groups for which Theorem \ref{thm:total} is most natural and we demonstrate this application by studying first order PDE under point-transformations. In the final section we give a new, constructive and intuitive proof of the Lie-Tresse theorem. This version of the theorem also gives general upper bounds on the minimal number of generating invariants and we consider a few examples of this bound. In the Appendix we consider the fundamental concept from \cite{OP-2009} of \emph{persistence of freeness} of Lie pseudo-groups. This property of pseudo-groups was proven in \cite{OP-2009} by rather complex means, but we show that it becomes essentially trivial in our framework.

\section{Lie pseudo-groups and moving frames}\label{sec:LG}

This section gives a rapid introduction to the formalism and constructions of the groupoid approach to Lie pseudo-groups, \cite{OP-2005}, and the equivariant moving frame, \cite{OP-2008}. The setup mirrors that of the first paper \cite{IMF}. 

\subsection{Pseudo-groups}
Let $\E\to\X$ be a smooth fiber bundle with fibers $\U$ and local coordinates $z=(x,u)$, $x\in\R^n$, $u\in\R^m$. Our results are mostly local and it will essentially suffice to consider $\E=\X\times\U\to\X$ where $\X\subset\R^n$ and $\U\subset\R^m$ are open sets. Regarding the local coordinates, we shall call the $x$-coordinates \emph{horizontal source coordinates} and the $u$-coordinates \emph{vertical source coordinates}. The \emph{diffeomorphism} pseudo-group on $\E$, $\D(\E)$ (or $\D$ when $\E$ is implied), is the collection of all (locally defined) invertible smooth maps, $\p$, between open sets in $\E$. We abuse notations slightly and write $\p:\E\to\E$ for these maps, even though they are only locally defined. This convention will hold throughout this paper and will apply to all locally defined objects.

The pseudo-group $q$-jets $j^q\p\resz$, $\p\in\D$, form a groupoid $\D_q$ fibered over $\E$ (for all $q\leq\infty$) with source and target maps (let $j^q\p\resz=(z, Z, \ldots, Z^a_A, \ldots)\in\D_q$)
\[
\sigma(j^q\p\resz)=z,\quad \tau(j^q\p\resz)=Z
\]
and groupoid operation $j^q\psi\restrict{Z}\cdot j^q\p\resz=j^q(\psi\comp\p)\resz$ where the composition is the standard composition operation on jets and the source of $j^q\psi\restrict{Z}$ is equal to the target of $j^q\p\resz$. Each $\psi\in\D$ acts on the groupoid elements $\D_q$ that have source within the domain of definition of $\psi$ by
\beq\label{eq:act}
R_\psi\cdot j^q\p\resz:=j^q\p\resz\cdot j^q\psi^{-1}\restrict{\psi(z)}.
\eeq
Note that $R_\psi$ agrees with the action $z\to\psi(z)$ on $\E$. For $\p\in\D$ we denote its jets by
\[
j^q\p\resz=(z, Z, \ldots, Z^a_A, \ldots)_{|A|\leq q},
\]
where $1\leq a\leq n+m$ and $A\in\Nz^{n+m}$ is a multi-index. We define the vector of jets $Z^{(q)}$ by $(z, Z^{(q)}):=(z, Z, \ldots, Z^a_A, \ldots)_{|A|\leq q}$.

\begin{Rem}
In the following, in lieu of explicit mention, the lower case Latin letters, $i,j,k,l$ will denote indices between $1$ and $n$, while lower case Greek letters, $\alpha, \beta$, etc. will denote indices between $1$ and $m$. The upper case Latin letters $J,K,L$ denote multi-indices in $\Nz^n$. The letters $a, b, c$ are reserved for superscripts in the jets $Z$ and run from $1$ to $m+n$ and their upper case counterparts $A,B,C$ denote multi-indices in $\Nz^{n+m}$.
\end{Rem}

The contact 1-forms on $\D_\infty$ (those 1-forms that are annihilated by the (pull-back of the) maps $z\mapsto j^\infty\p\resz$) have a basis that is invariant under the action (\ref{eq:act}). These forms are called the Maurer-Cartan forms of $\D$. They are denoted $\mu^a_A$ and the finite collections $\{\mu^a_A\}_{|A|<q}$ are a basis of the contact 1-forms on $\D_q$. Their structure equations are
\beq\label{eq:Dstr}
d\mu^a_A=\sum_{1\leq b\leq n+m}\o^b\wedge\mu^a_{A, b}+\sum_{\substack{B+C=A \\ |C|\geq1}}\sum_{1\leq b\leq n+m}\mu^a_{B, b}\wedge\mu^b_C,
\eeq
where $\o^b:=\sum_cZ^b_cdz^c$ are invariant forms on the base $\E$. See \cite{OP-2005}

Now let $\H$ be a sub Lie pseudo-group of $\D(\E)$ with sub groupoids $\H_p\subset\D_p$. This means that in local coordinates the elements of $\H$ are local solutions to a system of PDEs (called the \emph{determining, or defining, equations}),
\beq\label{eq:def}
F^{(t)}(z, Z^{(t)})=0 
\eeq
where each expression in $F^{(t)}$ is assumed to be genuinely $t\th$ order and full rank in the $t\th$ order jets. We assume this system is \emph{formally integrable} and real analytic and therefore locally solvable. Also, we assume that (\ref{eq:def}) is a complete system for $\H$, in the sense that
\[
\H_p=\left\{F^{(t)}=0~|~ t\leq p\right\}.
\]

\begin{Rem}\label{rem:order}
The \emph{order} of the Lie pseudo-group $\H$ is the smallest number $t^*$ such that all determining equations are differential consequences of the system $\left\{F^{(t)}=0~|~ t\leq t^*\right\}$. The important property of $t^*$ is that any local solution to $\left\{F^{(t)}=0~|~ t\leq t^*\right\}$ will define a pseudo-group element, as that solution will solve all the determining equations. 
\end{Rem}
%
%

 Linearizing (\ref{eq:def}) at the \emph{identity section} $\one^t\resz$ (the $t$-jet of the identity solution $z\mapsto z$) gives the linear system
\[
L^{(t)}(z, Z^{(t)})=\sum_{\substack{Z^a_A\\ 0\leq|A|\leq t}}\left(\frac{\partial F^{(t)}}{\partial Z^a_A}\restrictbig{\one^t\resz}\right)Z^a_A.
\]  
and restricting (or pulling back) the Maurer-Cartan forms $\mu^a_A$, to $\H_t$ will introduce linear dependencies among them precicely prescribed by these linear equations such that on $\H_p$ we have (see \cite{OP-2005})
\[
\sum_{\substack{Z^a_A\\ 0\leq|A|\leq t}}\left(\frac{\partial F^{(t)}}{\partial Z^a_A}\restrictbig{\one^t\resz}\right)_{z\leftrightarrow Z}\mu^a_A=0,
\] 
where, in the coefficients of $\mu^a_A$, we have replaced source coordinates $z$ by target coordinates $Z$.

By prolongation, $\H$ acts on the (submanifold) jet-bundles $J^q(\E)$, whose elements we denote by $j^q u\resx=(x, u,\ldots, u^\alpha_J,\ldots)$, $1\leq\alpha\leq m$ and $J\in\Nz^n$. Recall the total derivative on $J^\infty(\E)$, 
\[
D_i=\pder{x_i}+\sum_{\beta, K}u^\beta_{K,i}\pder{u^\beta_K},
\]
s.t. $u^\alpha_J=D^Ju^\alpha$. We denote the target coordinates of the prolonged action with capitalized, hat-wearing, letters:
\beq\label{eq:hatw}
j^\infty\p\resz\cdot j^\infty u\resx=(X, U, \ldots, \Uh^\alpha_J, \ldots),
\eeq
where $j^\infty\p\resz$ and  $j^\infty u\resx$ have the same source coordinate in $\E$, under $\H_\infty\to\E$ and $J^\infty(\E)\to\E$, respectively.

The pull-back of $\H_q\to\E$ by the canonical projection $J^q(\E)\to\E$, along with the prolonged action, gives the doubly fibered space
\[
\begin{tikzpicture}
  \node (A0) at (0,0) {$\Ht_q$};
  \node (A1) at (240:2cm) {$J^q(\E)$};
  \node (A2) at (300:2cm) {$J^q(\E)$,};
  \draw[->,font=\scriptsize]
  (A0) edge node[left] {$\tilde{\sigma}$} (A1)
  (A0) edge node[right] {$\tilde{\tau}$} (A2);

\end{tikzpicture}
\] 
where, in local coordinates,
\[
\tilde{\sigma}(j^q\p\resz, j^q u\resx)=j^qu \resx,\quad \tilde{\tau}(j^q\p\resz, j^q u\resx)=j^q\p\resz\cdot j^q u\resx.
\]
The pseudo-group $\H$ acts on the bundles $\Ht_q$ by \emph{right-regularization}, by
\beq\label{eq:rr}
\psi\cdot (j^q\p\resz, j^q u\resx):=(j^q\p\resz\cdot j^q\psi^{-1}\restrict{\p(z)}, j^q\psi\resz\cdot j^q u\resx)
\eeq
where the domain of definition of $\psi$ includes $z$. Importantly, the target map $\tilde{\tau}$ is invariant under this action and so the pull-back of any differential form of $J^q(\E)$ under $\tilde{\tau}$ is invariant on $\Ht_q$. In particular, the \emph{lifted invariants} are the pull-backs of the coordinate functions on $J^\infty(\E)$, $\tilde{\tau}(\uaJ)=\UhaJ$. The exterior derivative of lifted invariants is given by the fundamental \emph{recurrence formula},
\beq\label{eq:rec}
d\UhaJ=\sum_i^n\UhaJi\oi+\tilde{\tau}^*\left(\pder[\widehat{U}\av_J]{Z^a_A}\resone\right)\mu^a_A+\text{cont.}
\eeq
where ``cont.'' are contact forms on the bundle $\Jinf(\E)$ and automatically vanish when we restrict to a section of $\E$. Notice that, in the above formula, the first sum is only taken over the first $n$ base forms $\o^i$, $1\leq i\leq n$.

It often happens that, when constructing a moving frame (see below), we find a lifted invariant $\UhaJ$ that depends on the (submanifold) jets $J^q(\E)$ but on the pseudo-group parameters in $\H_p$, for $q>p$. To accommodate this feature, we define the following spaces. Let $q>p$ and consider the pull-back of $\H_p\to\E$ by $J^q(\E)\to\E$. Call this space $\Ht_{p,q}$. The canonical projection $\H_q\to\H_p$ induces a projection $\Ht_q\to\Ht_{p,q}$. The local coordinates of $\Ht_{p,q}$ are
\[
(j^p\p\resz, j^q u\resx)
\]
and the the pseudo-group action $\H$ on $\Ht_q$ reduces to an action of $\H$ on $\Ht_{p,q}$,
\beq\label{eq:rrpq}
\psi\cdot (j^p\p\resz, j^q u\resx):=(j^p\p\resz\cdot j^p\psi^{-1}\restrict{\p(z)}, j^q\psi\resz\cdot j^q u\resx).
\eeq

\subsection{Equivariant moving frames}

A \emph{partial moving frame} $\Bt_{p,q}$ is any smooth $\H$-invariant subspace of $\Ht_{p,q}$ under the right-regularized action (\ref{eq:rrpq}). We refer to the $p$ as the \emph{group order} of the partial moving frame and $q$ as its \emph{jet order}. The domain of definition of $\Bt_{p,q}$ is the projection 
\[
\S_q=\tilde{\sigma}\left(\Bt_{p,q}\right)\subset J^q(\E)
\]
where $\tilde{\sigma}$ is the source map in $\Ht_q$. Note that $\S_q$ is an invariant set. 

%
%

In practice, a partial moving frame is constructed as the level set of some of the lifted invariants $\UhaJ$, viewed as functions on some $\Ht_{p,q}\to\R$. The recurrence formula (\ref{eq:rec}) then tells us what linear dependencies among the Maurer-Cartan forms are introduced by restricting them to $\Bt_{p,q}$.


Let $\Bt_{p,q}\to\S_q$ be a partial moving frame with domain of definition $\S_q$. Ultimately we are interested in the local congruence of (local) sections of $\E\to\X$ under $\H$. For a section $S$ of $\E$ denote the jets of $S$ by $j^qS$ and assume $im(j^qS)\subset\S_q$. Pulling the bundle $\Bt_{p,q}\xrightarrow{\tilde{\sigma}}\S_q$ back by $j^qS$ will give a bundle\footnote{We continue to use the convention that functions and sections need only be defined on subsets of their domain space, and we extend this convention to bundles.} $\B^{j^qS}_p\to \X$. This is a subbundle of the pull-back bundle $S^*(\H_p\to\E)$ over $\X$.

\begin{Rem}\label{rem:emb}
Note that since $S:\X\to\E$ is one-to-one, $S^*(\H_p\to\E)$ and $\B^{j^qS}_p\to\X$ are subbundles of $\H_p\to\E$ under an embedding we denote $\Sb$.
\end{Rem}

The solution to the congruence problem for \emph{horizontal actions} in \cite{IMF} rested on results on the equivalence of sections of the bundle $S^*(\H_p\to\E)$. However, the horizontal case was made simple by the fact that $S^*(\H_p\to\E)$ was isomorphic to the groupoid $\G_p\to\X$ of a certain isomorphic Lie pseudo-group $\G\simeq\H$. (In the language of \cite{S-2000} $\H$ was a one-to-one prolongation of $\G$.) In the following we shall write $\H^S_p\to\X$ instead of $S^*(\H_p\to\E)$.

Unfortunately, in the non-horizontal case the bundle $\H^S_p\to\X$ is no longer isomorphic to the groupoid $\G_p\to\X$ of some known Lie pseudo-group as in \cite{IMF}. However, the proof of the key theorem \cite[Theorem 3.12]{IMF} carries over, word-for-word, to prove the following modification. A quick comment on notation: As in the \cite{IMF}, let $\H_p\irestrict{Z_0}\to\E$ denote the subbundle of $\H_p\to\E$ given by $\tau^{-1}(Z_0)$. This is the bundle we obtain when we normalize all zero order lifted invariants to $Z_0$ in the process of building a moving frame. Assuming we can normalize all zero order lifted invariants is not necessary but does simplify the discussion (see \cite{IMF}).

\begin{theorem}\label{GcoX}
Let $\H$ be a Lie pseudo-group, of order $t^*$, of transformations on the base manifold $\E$, and let $S_1$ and $S_2$ be sections of $\E$. Let $s_j$ be a section of $\H^{S_j}_p\irestrict{Z_0}\to \X$ for $j=1,2$, and $p\geq t^*$. A local transformation $f:\X\to \X$ preserves all the pulled-back Maurer-Cartan forms (recall Remark \ref{rem:emb} for the definition of $\Sb$),
\beq\label{eq:GcoX}
f^*s_2^*\Sb_2^*\mu^a_A=s_1^*\Sb_1^*\mu^a_A,\quad |A|<p,
\eeq
if and only if there exists a section $f_p$ of $\H^{S_1}_p\to \X$ such that
\[
f_p^*\Sb_1^*\mu^a_A=0,\quad |A|<p,
\]
and $R_{\Sb_1\comp f_p}\cdot(\Sb_1\comp s_1)=\Sb_2\comp s_2$. 
\end{theorem}

\begin{Rem}
Notice that the above theorem is only true when $p$ is at least as great as the order of the Lie pseudo-group $\H$. This is because the proof relies on local solvability in the groupoid $\H_p$, i.e. for each point solution to
\[
F^{(t)}=0, \quad t\leq p
\]
there must be a pseudo-group element passing through that point. See \cite{IMF} for details. 
\end{Rem}

Notice that $\Sb_1\comp f_p$ annihilates the contact codistribution of $\H_p$ of which the Maurer-Cartan forms $\mu^a_A$ are a basis and so there is hope that $\Sb_1\comp f_p$ can be \emph{integrated} to a bona-fide pseudo-group element of $\H$. That is, $\Sb_1\comp f_p$ could be given as an initial condition for the system of determining equations of $\H$. We make a general definition of the pseudo-groups for which this is possible, but hold off on characterizing convenient subclasses of these further until Section \ref{sec:qhpg} below. 

\begin{definition}
Let $\H$ be a Lie pseudo-group, of order $t^*$, of transformations on the bundle $\E\to\X$. For $r\geq t^*$, the pseudo-group $\H$ is called \emph{$r\th$ order vertically integrable} if, for each section $S$ of $\E$, and each $p\geq r$, every map $f_p:\X\to\H^S_p$ such that $\Sb\comp f_p$ annihilates all Maurer-Cartan forms $\mu^a_A$, $|A|<p$ on $\H_p\to\E$ and each point $f_p(x)$ on the image of $f_p$, there exists a neighborhood $V$ of $S(x)\in\E$  and a pseudo-group element $\p\in\H$ defined on $V$ such that $j^p\p\comp S=\Sb\comp f_p$ on $S^{-1}(V)$, and $r$ is the smallest such number.
\end{definition}  

For vertically-integrable pseudo-groups we have the following consequence of the previous theorem.

\begin{theorem}\label{GcoX2}
Let $\H$ be a $r\th$ order vertically integrable Lie pseudo-group of transformations on $\E\to\X$, and let $S_1$ and $S_2$ be sections of $\E$. Let $s_j$ be a section of $\H^{S_j}_p\irestrict{Z_0}\to \X$ for $j=1,2$ and $p\geq r$. A local transformation $f:\X\to \X$ preserves all the pulled-back Maurer-Cartan forms,
\beq\label{eq:GcoX}
f^*s_2^*\Sb_2^*\mu^a_A=s_1^*\Sb_1\mu^a_A,\quad |A|<p,
\eeq
if and only if there is a $\p\in\H$ such that $\p\comp S_1=S_2$.
\end{theorem}

Due to Theorem \ref{GcoX2} the involutive moving frame \cite{IMF} completely carries over to these pseudo-groups, including the following fundamental theorem on congruence in $\E$ under $\H$, which has, verbatim, the same proof as in the horizontal case, \cite[Theorem 4.15]{IMF}.

\begin{theorem}\label{thm:total}
Let $S_1$ and $S_2$ be sections of $\E$ and let $\H$ be an $r\th$ order vertically integrable Lie pseudo-group of transformations on $\E$. Let $\B_{p,q}\to\S_q$ be a partial moving frame with $p\geq r$, on which $U=\tilde{\tau}^*(u)=\text{constant}$. Then the following are equivalent.
\begin{enumerate}
\item $S_1$ and $S_2$ are locally congruent under some $\p\in\H$.
\item For each section $s_1$ of $\B^{j^qS_1}_p\to \X$ there is a section $s_2$ of $\B^{j^qS_2}_p\to \X$  and a map $f$ such that
\[
f^*s_2^*\Sb_2^*\mu^a_A=s_1^*\Sb_1\mu^a_A,\quad |A|<p.
\]
\item There exists a pair of sections, $s_j$ of $\B^{j^qS_j}_{p}\to \X$, $j=1,2$, and a map $f$ such that
\[
f^*s_2^*\Sb_2^*\mu^a_A=s_1^*\Sb_1\mu^a_A,\quad |A|<p.
\]
\end{enumerate}
\end{theorem}

This theorem presents us with two problems. First we must solve the coframe equivalence problem in part 3 of Theorem \ref{thm:total} to obtain a function $f:\X\to\X$ and its corresponding $f_p$ of Theorem \ref{GcoX}. Then we must integrate $f$ and $f_p$ to a pseudo-group element $\p\in\H$. In this paper we shall focus on those vertically integrable pseudo-groups where the section $f_p$ can be \emph{uniquely} integrated to a pseudo-group element $\p\in\H$. We shall call these pseudo-groups \emph{quasi-horizontal} and we now turn to characterizing them in terms of their determining equations. Note that for these quasi-horizontal pseudo-groups it is enough to find the $f$ of part 3 of Theorem \ref{thm:total} in order to completely solve the equivalence problem. Hence, the involutive moving frame method from \cite{IMF} will completely carry over to the quasi-horizontal case. 


\subsection{Quasi-horizontal Lie pseudo groups}\label{sec:qhpg}

We now turn to the characterization of Lie pseudo-groups for which the sections $f_p$ of Theorem \ref{GcoX} can be uniquely integrated to pseudo-group elements. 

%

\begin{definition}\label{def:quasi}
Let $\H$ be a Lie pseudo-group, of order $t^*$, of transformations on $\E$ with determining equations $F^{(t)}(t, Z^{(t)})=0$. Given local coordinates $z=(x,u)$ on $\E$, the elements of $\H$ are local diffeomorphisms $(x,u)=z\mapsto Z=(X,U)$, on $\E$. Then $\H$ is called \emph{quasi-horizontal} if there exists an $r\geq t^*$ such that the following holds.

In the determining equations of $\H$, each derivative of order $r$, $Z^a_A$, $|A|=r$, which has at least one $\ua$ derivative can be taken as principal. In other words, for each $A\in\mathbb{Z}_{\geq0}^{r-1}$, $1\leq a\leq n+m$ and $1\leq\alpha\leq m$, one of the determining equations $F^{(r)}=0$ for $\H$ reads
\[
\frac{\partial^{r}Z^a}{\partial\ua\partial z^A}=F^{a}_{\alpha, A}(z, Z^{(r)}).
\]

Notice that if the two conditions hold for some $r$, they also hold for all $r'\geq r$. We say $\H$ is \emph{quasi-horizontal of order} $r$ if $r$ is the smallest number for which the above holds. 
\end{definition}


\begin{Rem}\label{rem:quasi}
Importantly, if $\H$ is quasi-horizontal in variables $(x, u)$ it is also quasi-horizontal (with the same horizontal order) after any fiber preserving change of variables in $\E$, $(x, u)\mapsto (y(x), w(x, u))$. The above definition is therefore consistent on a general bundle $\E\to\X$.  
\end{Rem}

\begin{Rem}
Horizontal pseudo-groups, as defined in \cite{IMF}, are trivially quasi-horizontal.
\end{Rem}

 
Now let $f:\X\to\X$ and $f_p:\X\to\H^{S_1}_p$ satisfy the conditions of Theorems \ref{GcoX} and \ref{thm:total} where $\H$ is quasi horizontal (of order $r\leq p$) and assume the images of $S_j$ are the graphs of two functions $u_j:\X\to\U$, $j=1,2$. We can straighten $S_1$ out in the $x$-direction by a (fiber-preserving) change of variables,
\[
(y,w)\mapsto (y, u_1(y)+w).
\]
We call the coordinates $(y, w)$ \emph{flat} for $S_1$, and by Remark \ref{rem:quasi} $\H$ is also quasi-horizontal in $(y, w)$. Notice that on $S_1$ we have
\beq\label{eq:dy}
\pder{y}=\pder{x}+u_x\pder{u}.
\eeq
In the flat coordinates $S_1$ is the zero section $y\to (y, 0)$ and we shall treat $\Sb_1\comp f_p(y, 0)$ as initial conditions for a sequence of normal systems of partial differential equations for the vector 
\[
\pmb{Z}:=Z^{(p-1)}=(Z^1, \ldots, Z^a_{A},\ldots)\restrict{|A|\leq p-1},
\] 
where $A=(A_1, \ldots, A_{n+m})\in\Nz^{n+m}$ now denote derivatives w.r.t. $(y, w)$. The $p\th$ order determining equations for $\H$ involve some equations that only contain $y$-derivatives of $\pmb{Z}$, but these equations are satisfied by $\Sb_1\comp f_p(y, 0)$ by construction. The rest of the $p\th$ order equations involve at least one $w$-derivative. Since $\H$ is quasi-horizontal, using the standard trick of turning a PDE of arbitrary order into a first order system, we can write these latter equations as
\beq\label{eq:Zb}
\pder{w^1}\pmb{Z}=\pmb{F}_1(y, w, \pmb{Z},\pder{y_1}\pmb{Z},\ldots, \pder{y_n}\pmb{Z}) 
\eeq
for some function $\pmb{F}_1$.

Now, given the map $\Sb_1\comp f_p(y, 0)$ consider the initial value problem (\ref{eq:Zb}) with initial value $\pmb{Z}(y, 0)=\Sb_1\comp f_p(y,0)$. This is obviously a normal system and uniquely integrable to an open set in $\R^{n+1}\times\{0\}^{m-1}$ by the Cauchy-Kovalevskaya theorem. We can then use the newly integrated $\pmb{Z}(y, w^1, 0)$ as an initial condition for a system
\[
\pder{w^2}\pmb{Z}=\pmb{F}_2(y, w, \pmb{Z},\pder{y_1}\pmb{Z}, \pder{y_n}\pmb{Z}) 
\]
and so on until we have integrated $\pmb{Z}$ to the prolongation of a pseudo-group element $\p\in\H$. This proves the following.

\begin{theorem}\label{thm:quastotal}
Let $\H$ be quasi horizontal of order $r$ and let $f$ and $f_p$ satisfy the conditions of Theorems \ref{GcoX} and \ref{thm:total} with $p\geq r$. Then $f$ can be uniquely integrated to a pseudo-group element $\p\in\H$. 
\end{theorem}

Theorem \ref{thm:quastotal} allows us to transfer the equivalence method for horizontal actions from \cite{IMF} to the quasi horizontal case as outlined in the comments following Theorem \ref{thm:total}. We now demonstrate the use of Theorem \ref{thm:total} by studying the equivalence problem of first order partial differential equations.

\subsection{Equivalence of first order PDE}\label{sec:1stpde}

Consider a general first order partial differential equation for a scalar function $u(x, y)$ of two variables,
\[
u_y=f(x, y, u, u_x).
\]
We shall study the equivalence of such equations under the pseudo-group, $\G$, of \emph{point transformations}. 

\begin{Rem}
We shall not go into too much detail regarding the actual, nitty-gritty, computations in this example. A forthcoming paper is dedicated to the computational advantages of the involutive moving frame.
\end{Rem}

Writing $u_x=p$ and $u_y=q$, the elements of the pseudo-group are local diffeomorphisms
\[
(x, y, u, p, q)\overset{\p}{\mapsto} (X, Y, U, P, Q).
\]
Since the above is the first prolongation of $(x, y, u)\mapsto(X, Y, U)$, we have $X_p=X_q=Y_p=Y_q=U_p=U_q=0$ and
\beq\label{eq:pqmat}
\bbm P \\ Q \ebm=\bbm X_x+pX_u & X_y+qX_y \\ Y_x+pY_u & Y_y+qY_u \ebm^{-T}\bbm U_x+pU_u \\ U_y + qU_u\ebm.
\eeq
These are the basic first order equations for $\G$ from which all other defining equations are found, by prolongation and projection. Note that differentiating both sides of (\ref{eq:pqmat}) w.r.t. $q$ will give determining equations $P_q=\cdots$ and $Q_q=\cdots$ where the right hand sides are first order, and we easily determine that this Lie pseudo-group is quasi horizontal of horizontal order 1. 

Since the entries of the matrix above appear so frequently throughout we denote them by
\[
\bbm g_{11} & g_{12} \\ g_{21} & g_{22} \ebm:=\bbm X_x+pX_u & X_y+qX_y \\ Y_x+pY_u & Y_y+qY_u \ebm.
\]  
We normalize $(X, Y, U, P, Q)$ to zero so that
\[
U_x=-pU_u,\quad U_y=-qU_u
\]
and the recurrence formula gives $\mu^u_X=-\o^p$ and $\mu^u_Y=-\sum \widehat{Q}_i\oi$. The recurrence formula for the first order lifted invariants $\widehat{Q}_i$ gives
\[
\aligned
d\widehat{Q}_X&=\sum \widehat{Q}_{Xi}\oi +\o^p \widehat{Q}_U-\widehat{Q}_P \mu^p_X+\widehat{Q}_P \widehat{Q}_X \mu^y_X+\widehat{Q}_X \mu^u_U-\widehat{Q}_X \mu^x_X-\widehat{Q}_X\mu^y_Y-\widehat{Q}_Y \mu^y_X+\mu^p_Y,\\
d\widehat{Q}_Y&=\sum \widehat{Q}_{Yi}\oi+\o^q \widehat{Q}_U-\widehat{Q}_P \mu^p_Y+\widehat{Q}_P \widehat{Q}_Y \mu^y_X+\widehat{Q}_Y\mu^u_U-\widehat{Q}_X \mu^x_Y-2 \widehat{Q}_Y \mu^y_Y+\mu^q_Y,\\
d\widehat{Q}_U&=\sum \widehat{Q}_{Ui}\oi-\widehat{Q}_P \mu^p_U+\widehat{Q}_P \widehat{Q}_U\mu^y_X-\widehat{Q}_X \mu^x_U-\widehat{Q}_U \mu^y_Y-\widehat{Q}_Y \mu^y_U+\mu^q_U,\\
d\widehat{Q}_P&=\sum \widehat{Q}_{Pi}\oi+ \mu^x_X+\widehat{Q}_P^2\mu^y_X-\widehat{Q}_P \mu^y_Y-\mu^x_Y,
\endaligned
\]
where, in the sums above, $i$ runs through the lifted horizontal base coordinates $X, Y, U, P$. We normalize $\widehat{Q}_i=0$ for $i=X, Y, U, P$ to obtain
\[
\aligned
\mu^p_Y&=-\sum \widehat{Q}_{Xi}\oi,\\
\mu^q_Y&=-\sum \widehat{Q}_{Yi}\oi,\\
\mu^q_U&=-\sum \widehat{Q}_{Ui}\oi,\\
\mu^x_Y&=-\sum \widehat{Q}_{Pi}\oi.
\endaligned
\]
Also, the recurrence formula $0=d\widehat{Q}=\sum \widehat{Q}_i\oi+\mu^u_Y$ now becomes $\mu^u_Y=0$. 

The next step in the equivalence procedure is to compute the structure equations $d\oi$ on our equivariant moving frame $\Bt_{1,1}$. Doing so we encounter as structure coefficients the lifted second order invariants $\widehat{Q}_{ij}$. We manage to normalize all of them except for $\widehat{Q}_{PP}$ to zero. At this point, there are $r^{(2)}=13$ second order group parameters we have \emph{not} solved for. These 13 parameters along with $X_x, X_u, Y_x, Y_y, Y_u, U_u, P_u, P_x$ and the second order jets $j^2q\restrict{(x,y,u,p)}$ parametrize the equivariant moving frame $\Bt_{2,2}$.

Meanwhile,
\[
\widehat{Q}_{PP}=\frac{g_ {11}^2 q_{pp}}{U_u(g_{22}-g_{21} q_p)}
\]
and so this equivalence problem branches at this point, depending on whether $q_{pp}=0$ or not. 

\subsubsection{Branch 1}
When $q_{pp}=0$ the original PDE has the form
\[
u_y=g(x, y, u)u_x+h(x, y,u). 
\]
On this branch $\widehat{Q}_{PP}=0$ and there are no more normalizable structure coefficients in $d\oi$. We must then check for involutivity of the $\oi$. We have the structure equations
\[
\aligned
d\o^x&= -\o^x\wedge\mu^x_X-\o^u\wedge\mu^x_U,\\
d\o^u&= -\o^x\wedge\mu^y_X-\o^u\wedge\mu^y_Y-\o^u\wedge\mu^y_U,\\
d\o^u&= \o^x\wedge\o^p-\o^u\wedge\mu^u_U,\\
d\o^p&= -\o^x\wedge\mu^p_X-\o^u\wedge\mu^p_U-\o^p\wedge\mu^u_U+\o^p\wedge\mu^x_X.
\endaligned
\]
To perform Cartan's test for involutivity we compute the reduced Cartan characters $s^{(1)}_1, s^{(1)}_2, s^{(1)}_3, s^{(1)}_4$ for the above system. The first Cartan character $s^{(1)}_1$ is the maximum rank of the matrix
\[
\left(
\begin{array}{cccccccc}
 -a & 0 & 0 & -c & 0 & 0 & 0 & 0 \\
 0 & -a & -b & 0 & -c & 0 & 0 & 0 \\
 0 & 0 & 0 & 0 & 0 & -c & 0 & 0 \\
 d & 0 & 0 & 0 & 0 & -d & -a & -c \\
\end{array}
\right)
\]
where $a, b, c, d$ are arbitrary real numbers. This maximum rank is achieved for example when $a=c=1$ and $b=d=0$ and is $s^{(1)}_1=4$. Some simple computations reveal that $s^{(1)}_2=3$, $s^{(1)}_3=1$ and $s^{(1)}_4=0$. Now, Cartan's test for involution asks whether
\[
13=r^{(2)}=\sum_i i\cdot s^{(1)}_i=1\cdot 4+2\cdot 3+3\cdot1+4\cdot0=13.
\]
Since this is true, the frame $\{\o^x,\o^u,\o^u,\o^p\}$ is involutive and \emph{all} first order PDE of the form
\[
u_y=g(x, y, u)u_x+h(x, y,u)
\]
are equivalent under a point transformation and the general equivalence map depends on one arbitrary function of three variables, as per Cartan's theory of involution. A canonical form of a PDE in this class is the trivial $u_y=0$.

\subsubsection{Branch 2}
When $q_{pp}\neq0$ we normalize $U_u$ to obtain
\[
U_u= \frac{g_ {11}^2 q_{pp}}{g_{22}-g_{21} q_p}
\]
and the recurrence formula gives
\[
\mu^u_U=\sum \widehat{Q}_{PPi}\oi+2 \mu^x_X-\mu^y_Y.
\]
Going back to the structure equations $d\oi$ on $\Bt_{1,2}$ and studying the structure coefficients, we find that 
\[
d\widehat{Q}_{PPP}=\sum \widehat{Q}_{PPPi}\oi-\mu^x_X \widehat{Q}_{PPP}+\mu^y_Y \widehat{Q}_{PPP}+3 \mu^x_U+3 \mu^y_X
\]
and that we can normalize $\widehat{Q}_{PPP}=0$ and solve for $X_u$. The explicit formula for this normalization is 
\[
X_u= \frac{g_{11} g_{21} q_{pp}}{g_{21} q_p-g_{22}}-\frac{g_{11} \
q_{ppp}}{3 q_{pp}}.
\]
At the next stage, computing $d\oi$ on $\Bt_{1,3}$ we find the structure function $\widehat{Q}_{PPPP}$ and the recurrence formula gives
\[
d\widehat{Q}_{PPPP}=\sum \widehat{Q}_{PPPPi}\oi+\frac{2\widehat{Q}_{PPPP}}{3} (\mu^x_X -\mu^y_Y)-2 \mu^y_U.
\]
Therefor $Y_u$ can be normalized from $\widehat{Q}_{PPPP}=0$. The formula we obtain is
\[
Y_u= \frac{-9 g_ {21}^2 q_ {pp}^4+6 g_{21} q_{ppp} q_ {pp}^2 (g_{21} \
q_p-g_{22})-3 q_{pppp} q_{pp} (g_{22}-g_{21} q_p)^2+4 q_ {ppp}^2 \
(g_{22}-g_{21} q_p)^2}{18 q_ {pp}^3 (g_{22}-g_{21} q_p)}.
\]
At the next stage, computing $d\oi$ on $\Bt_{1,4}$, we find $\widehat{Q}_{5P}$ and the recurrence formula tells us that 
\[
d\widehat{Q}_{5P}=\sum \widehat{Q}_{5P,i}\oi-3\widehat{Q}_{5P}(\mu^x_X-\mu^y_Y)
\]
and some further computations reveal that 
\[
\widehat{Q}_{PPPPP}=-\frac{(40 q_{ppp}^3-45 q_{pp} q_{pppp} q_{ppp}+9 q_{pp}^2 q_{ppppp}) (g_{22}-g_{21} q_p)^3}{54 g_{11}^3 q_{pp}^6}.
\]
This branch of the equivalence problem therefore branches again at this juncture depending on the vanishing or non-vanishing of the expression
\[
40 q_{ppp}^3-45 q_{pp} q_{pppp} q_{ppp}+9 q_{pp}^2 q_{ppppp}.
\]

In the case this expression is identically zero, for example for the PDE
\[
u_y=u_x^2,
\]
then, after we finish normalizing all possible structure coefficients from $d\oi$, we find that we have normalized all second order group parameters except for the $r^{(2)}=4$ parameters $P_{uu}, P_{ux}, X_{ux}$ and $P_{xx}$. We again check for involution. This time we have the structure equations
\[
\aligned
d\o^x&=-\o^x\wedge\mu^x_X-\o^u\wedge\o^p+\o^u\wedge\mu^y_X,\\
d\o^y&=-\o^x\wedge\mu^y_X-\o^u\wedge\mu^y_Y,\\
d\o^u&=\o^x\wedge\o^p-2 \o^u\wedge\mu^x_X+\o^u\wedge\mu^y_Y,\\
d\o^p&=-\o^x\wedge\mu^p_X-\o^u\wedge\mu^p_U-\o^p\wedge\mu^x_X+\o^p\wedge\mu^y_Y.
\endaligned
\] 
This leads to reduced Cartan characters $4, 1, 0, 0$ and Cartans test gives
\[
r^{(2)}=4<1\cdot4+2\cdot1=6
\]
and so the coframe of $\oi$'s is not involutive. Prolonging to the next order we compute the structure equations $d\oi$ \emph{and} the exterior derivative of the not-yet normalized first order Maurer-Cartan forms $\mu^p_U,\mu^p_X,\mu^x_X,\mu^y_X,\mu^y_Y$. Normalizing all the lifted invariants appearing in this set of equations (notice however that the equation for $d\oi$ are unchanged) we normalize many third order group parameters and no new second order group parameters. After all is said and done we find that the ones we did not normalize are the $r^{(3)}=4$ third order group parameters
\[
P_{uuu}, P_{uux}, P_{uxx}, P_{xxx},
\] 
while the structure equations are
\[
\aligned
d\mu^p_U&=\o^x\wedge\mu^p_{UX}+\o^u\wedge\mu^p_{UU}+\o^p\wedge\mu^x_{UX}+\mu^p_U\wedge\mu^x_X-\mu^p_X\wedge\mu^y_X,\\
d\mu^p_X&=\o^x\wedge\mu^p_{XX}+\o^u\wedge\mu^p_{UX}+\o^q\wedge\mu^x_{UX}+\o^p\wedge\mu^p_U+\mu^p_X\wedge\mu^y_Y,\\
d\mu^x_X&=\o^u\wedge\mu^x_{UX}+\o^x\wedge\mu^p_U-\o^u\wedge\mu^p_X,\\
d\mu^y_X&=-\o^x\wedge\mu^x_{UX}+\o^u\wedge\mu^p_U-\mu^x_X\wedge\mu^y_X-\mu^y_X\wedge\mu^y_Y,\\
d\mu^y_Y&=\o^x\wedge\mu^p_U-2\o^u\wedge\mu^p_X-\o^p\wedge\mu^y_X.
\endaligned
\]
This time the reduced Cartan characters are $4,0,0,0,0$ and Cartan's test for involution is satisfied: $r^{(3)}=4=1\cdot4 +0\cdot2+0\cdot3+0\cdot4+0\cdot5$. We say that the equivariant moving frame $\Bt_{2,5}$ we have computed is \emph{involutive}. Notice that \emph{all} PDEs
\[
u_y=f(x,y,u,p)
\]
where $q_{pp}\neq0$ and
\[
40 q_{ppp}^3-45 q_{pp} q_{pppp} q_{ppp}+9 q_{pp}^2 q_{ppppp}=0
\]
are point-equivalent to $u_y=u_x^2$. 

We stop here, but mention that for PDE such that neither $\widehat{Q}_{PP}$ nor $\widehat{Q}_{PPPP}$ vanish, we continue by normalizing $\widehat{Q}_{PPPP}=1$ and solving for $g_{11}$, recompute structure equations and so on.

\begin{section}{The Lie-Tresse theorem}\label{sec:LT}

When constructing a (partial) moving frame for a quasi-horizontal pseudo-group $\H$ of order $r$ there are two possible outcomes as one navigates the different branches that emerge. The first possibility is that at some stage the Maurer-Cartan forms $\mu^a_A$, $|A|<p$, $r\leq p$, on our partial moving frame $\B_{p,q}\to\S_q$ are involutive, \cite{IMF}. In this case, usually all sections whose jets lie within the domain of definition $\S_q$ are locally congruent as there are rarely any invariants present (although there are exceptions, \cite[p. 370]{O-1995}). 

Alternatively one manages to normalize all pseudo-group parameters of some order $p$. The moving frame machinery is then continued where we investigate the structure functions of the coframe in a search for further normalizations of group parameters. On some branches these structure functions are all constant and no more normalizations are possible. In these cases, the sections of that branch are all equivalent and have a finite dimensional Lie group of symmetries. 

In this section we look at the case when we can normalize all pseudo-group parameters of order $\leq p$ to obtain some partial moving frame $\Bt_{p,q}\to\S_q$. Let $\S_\infty:=(\pi^\infty_q)^{-1}(\S_q)$. A differential invariant is a smooth $\H$-invariant function on $\S_\infty$. The \emph{Lie-Tresse theorem}, or \emph{Fundamental basis theorem}, first proved by Lie in the finite dimensional case, \cite[p.760]{L-1893}, and extended by Tresse to include pseudo-groups, \cite{T-1894}, states that the algebra (over $\R$) generated by all such differential invariants is \emph{finitely generated} in the following sense. There exist differential invariants $I_1, \ldots, I_t$ and invariant differential operators $\D_1, \ldots, \D_n$ such that \emph{all} differential invariants can be written as smooth functions of the $I_i$ and their derivatives w.r.t $\D_j$. In \cite{OP-2009} the authors gave a constructive, practical proof of this theorem, which nevertheless relied upon some complicated results on non-standard algebraic objects like ``eventual submodules'' and required the computation of a Groebner basis for a certain module. In this section we give a straightforward and simple proof that is more practical than the one in \cite{OP-2009} since once we have completed an equivalence problem, we may simply read a generating set of invariants off the final structure equations.   

\begin{Rem}
In the Appendix we show that freeness (i.e. all pseudo-group parameters can be normalized) of a Lie pseudo-group implies quasi-horizontality, and so we are necessarily in the quasi-horizontal case without having restricted the generality of our Lie-Tresse theorem below.
\end{Rem}

Again, let $\Bt_{p,q}\to\S_q$ be a moving frame where we have normalized all pseudo-group parameters of order $\leq p$ so $\Bt_{p,q}$ is actually diffeomorphic to its domain of definition $\S_q\subset J^q(\E)$. For reasons that become clear in a moment, assume $p=r$ (we can do this since all pseudo-group parameters can be normalized by persistence of freeness, see Appendix). 


Let us assume, for simplicity that $\H$ acts transitively on $\E$ and so we can normalize all zero order lifted invariants $Z^a$. Then, modulo contact forms on $\S_q$, the Maurer-Cartan forms $\mu^a_A$, $|A|<r$, corresponding to group parameters, once restricted to $B_{r, q}$, become linear combinations
\[
\mu^a_A\mapsto \sum_{l=1}^nI^a_{A;l}\o^l
\]
of the horizontal forms $\o^l$, where $I^a_{A;l}$ are differential invariants on $\S_q$. Note that $1\leq l\leq n$. The forms $\o^l$, $1\leq l\leq n$, are a basis of contact invariant horizontal forms in $J^\infty(\E)$. Let us denote their dual differential operators by $\D_l$. Each $\o^l$ ($\D_l$) is a linear combination of the $dx^l$ ($D_{x^l}$) with coefficients that are functions on $\S_q$. 

The equivalence problem for sections $S_1$ and $S_2$ of $\E$ (with jets $j^qS_j$ lying in $\S_q$) reduces, according to Theorem \ref{thm:total}, to the equivalence of the $\mu^a_A$, $|A|<r$, as pulled back to the spaces $\B^{j^qS_1}_{r}$ and $\B^{j^qS_2}_{r}$, but both are diffeomorphic to open sets in the $n$ dimensional base space $\X$. Let us  describe this reduced equivalence problem on $\X$. Denote the pull-back of $\o^i$ to $\B^{j^qS_\epsilon}_{r}$ by $\widetilde{\o}_\epsilon^i$ for $\epsilon=1,2$, and that of the invariants $I^a_{A; i}$ by $I^a_{\epsilon; A; l}$, but notice that $I^a_{\epsilon; A; l}$ is simply $I^a_{A; i}\comp j^qS_\epsilon$. A local map $f:\X\to\X$ preserves all the pulled-back Maurer-Cartan forms $\mu^a_A$, $|A|<r$, if and only if 
\[
f^*\widetilde{\o}_2^i=\widetilde{\o}_1^i\quad \text{and} \quad f^*I^a_{2; A; l}=I^a_{1; A; l}.
\]
{\'E}lie Cartan solved this \emph{(extended) coframe equivalence problem} and we briefly describe his solution, see \cite{O-1995} for more. Write the structure functions
\[
d\o^i=J^i_{j,k}\o^j\wedge\o^k,\quad 1\leq i, j, k\leq n,
\]
for certain differential invariants defined on $(\pi^{q+1}_q)^{-1}\left(\S_q\right)$. Actually, the $J^i_{j,k}$ are themselves functions of some of the $I$'s since the Maurer-Cartan structure equations on $\D$, (\ref{eq:Dstr}) give
\[
d\o^i=-d\mu^i=-\sum_b \o^b\wedge\mu^i_b
\]
and so any invariant $J^i_{jk}$ appearing above will have come from the restriction of $\mu^i_b$ to the moving frame and can be found among the $I^a_{A;l}$. Denote the dual differential operators to the coframes $\owt^i_\epsilon$ by $\widetilde{\D}^\epsilon_i$. Note that $\widetilde{\D}^\epsilon_i$ is obtained by evaluating the coefficients of $\D_i$ at $j^qS_\epsilon$.  The \emph{invariants} of the extended coframe $\{\owt_\epsilon^i, I^a_{\epsilon; A; l}\}$ are the $I^a_{\epsilon; A; l}$ and all their derivatives w.r.t. the $\widetilde{\D}^\epsilon_i$. Cartan proved that an equivalence map $f$ exists if and only if $f$ pulls the (infinite) collection of invariants $\{(\widetilde{\D}^2)^KI^a_{2; A; l}\}_{K\in \mathbb{Z}_{\geq0}^n, |A|<r}$ to $\{(\widetilde{\D}^1)^KI^a_{1; A; l}\}_{K\in \mathbb{Z}_{\geq0}^n, |A|<r}$. There is a finite process for determining when this happens which has to do with finding a finite subset of the invariants that has maximal rank (again, see \cite{O-1995}). 

%

Let us denote the collection of invariants and their derivatives by
\[
\C:=\{(\widetilde{\D})^KI^a_{A; l}\}_{K\in \mathbb{Z}_{\geq0}^n, |A|<r}.
\]
Note that the invariants in $\C$ are defined on $\S_\infty$. Before proceeding to the Lie-Tresse theorem we arrange the functions from $\C$ in a map
\beq\label{cfuncs}
c_\infty:=(I^a_{A; l},\ldots, \D^KI^a_{A; l},\ldots)_{|K|\leq \infty,~ 1\leq a\leq n+m,~ 1\leq l\leq n,~ |A|<r},
\eeq
where $c_\infty:J^{\infty}(\E)\to \R^{\infty}$. We shall prove that the fibers of the map $c_\infty$ are precisely the orbits of $\H$ on $\S_\infty$ so that any function, constant on these orbits, must be a function of $c_\infty$. Since $c_\infty$ contains only the invariants $\D^KI^a_{A; l}$ this will prove the Lie-Tresse theorem. 

%
%
%

\begin{theorem}[Lie-Tresse]\label{thm:LieTresse}
Let $\Bt_{r,q}\to\S_q$ be an equivariant moving frame for a quasi-horizontal pseudo-group $\H$ of order $r$ where we have normalized all the pseudo-group parameters of orders $\leq r$. The invariants $$\{I^a_{A; l}\}_{1\leq l\leq n,~ 1\leq a\leq n+m,~ |A|<r}$$ generate the differential invariant algebra of $\H$ on $(\pi^\infty_q)^{-1}(\S_q)$.
\end{theorem}

\begin{proof}
The proof is based on the fact that the invariants $$\{I^a_{A; l}\}_{1\leq l\leq n,~ 1\leq a\leq n+m,~ |A|<r}$$ determine the local congruence problem for sections in $\E$. 

Let $S_1$ and $S_2$ be two sections of $\E$ with infinte jets $j^\infty S_j\restrict{x_j}=z_j^\infty$ at some points $x_j$, $j=1,2$, and suppose that all invariants of the collection $\mathcal{C}$ agree at the $z_j^\infty$, i.e.
\beq\label{IJ}
R\restrict{z_1^\infty}=R\restrict{z_2^\infty},\quad \forall R\in\C.
\eeq
(Equivalently, $c_\infty(z_1^\infty)=c_\infty(z_2^\infty)$.)

First we show that (\ref{IJ}) implies that for each invariant $R\in\C$, $R\comp j^\infty S_1$ and $R\comp j^\infty S_2$ \emph{overlap} (have the same image set) on some neighborhoods of $x_1$ and $x_2$. (Note that, for simplicity, we write $R\comp j^\infty S_j$ even though $R$ only depends on the jets of $S_j$ up to some finite order.)

Pick $n$ invariants $A_1, \ldots, A_n\in \C$ such that $$d_HA_1\wedge\dots\wedge d_HA_n\restrict{z^\infty_j}\neq 0$$ where $d_H$ is the horizontal part of the exterior derivative on $J^\infty(\E)$. (Note that we may have to restrict the $z_j^\infty$ to a dense subset of $\S_\infty$ in order to make sure this is possible.)

The (invariant) derivative with respect to $A_i$, $D_{A_i}$, is a linear combination of the $\D_j$ with coefficients depending on $\D_kA_l$. For a multi-index $J=(J_1, \ldots, J_n)\in\mathbb{Z}_{\geq 0}^n$, we shall write $D_A^J$ for $D^{J_1}_{A_1}\dots D^{J_n}_{A_n}$. Therefore, by (\ref{IJ}), 
\[
(D_A^JR)\restrict{z_1^\infty}=(D_A^JR)\restrict{z_2^\infty}
\]
for all $R\in\C$ and $J\in \mathbb{Z}_{\geq 0}^n$.

Now take some invariant $R\in \C$ and consider the two maps $R_j:=R\comp j^\infty S_j$ after the change of variables $x\mapsto (A_1, \ldots, A_n)\comp j^\infty S_j\resx$. Note that $a_0:=(A_1, \ldots, A_n)\comp j^\infty S_1\restrict{x_1}=(A_1, \ldots, A_n)\comp j^\infty S_2\restrict{x_2}$. These maps have power series expansions around $a_0$ given by
\[
R_j=\sum\frac{1}{J!}(D_A^JR)\restrict{z^\infty_j}(a-a_0)^J.
\]

Therefore the $R_j$'s have the same power series coefficients and, everything being real-analytic, there must be open neighborhoods $V^R_j$ of $x_j$ such that 
\beq\label{eq:I1I2}
R_1(V^R_1)=R_2(V^R_2).
\eeq 

The next step in the proof is proving that $z_1^\infty$ and $z_2^\infty$ such that (\ref{IJ}) holds are congruent under $\H$. This actually follows quite easily from what we have shown so far. Recall that local congruence of $S_1$ and $S_2$, mapping $x_1$ to $x_2$ (and therefore $z_1^\infty$ to $z_2^\infty$) is equivalent to a local equivalence map (again, mapping $x_1\mapsto x_2$) between the extended coframes $\{ \owt^i_1, I^a_{1;A;l}\}$ and $\{ \owt^i_2, I^a_{2;A;l}\}$. Following Cartan's solution, there is a finite set of invariants, $R_1, \ldots, R_k\in \C$ such that there exists a local equivalence map between the extended coframes if and only if there are open neighborhoods $V_j$ of $x_j$, $j=1,2$, such that $R_t\comp j^\infty S_1(V_1)=R_t\comp j^\infty S_2(V_2)$ for all $t=1,\ldots, k$. Well, according to (\ref{eq:I1I2}) we can take
\[
V_j:=\bigcap_{t=1}^kV^{R_t}_j,\quad j=1,2,
\]
proving that $z_1^\infty$ and $z_2^\infty$ are congruent.

We have shown that the the fibers of the invariant map $c_\infty$ and the orbits of $\H$ on $J^\infty(\E)$ are the same thing. This completes the theorem. 

%
%
%

\end{proof}

\subsection{Upper bounds on the size of the generating set}

Theorem \ref{thm:LieTresse} provides an upper bound on the number of generating invariants. We have that 
\beq\label{genset}
\{I^a_{A; l}\}_{1\leq l\leq n,~ 1\leq a\leq n+m,~ |A|<r}
\eeq
is a set of generators, where
\[
\mu^a_A\mapsto I^a_{A; l}\o^l.
\]
Since those $\mu^a_A$ corresponding to principal derivatives in the determining equations for $\H$ are constant-coefficient linear combinations of parametric ones, we can focus only on the latter. Therefore, to count the number of $I$'s we simply have to count the number of parametric derivatives of order $<r$ in $\H$ and multiply by $n$. The former number is equal to the dimension of the groupoid $\H_{r-1}$. So a general upper bound for the \emph{minimal number of generating invariants} is
\[
n\cdot \left(\text{dim}\H_{r-1}\right).
\]
In the case that $\H$ is a finite dimensional Lie group its determining equations are maximally over-determined and so the dimension of the Lie group is equal to $\text{dim}\H_{r-1}$. We therefore have, for a Lie group $G$ acting freely on $\S_q$ the upper bound
\[
n\cdot \left(\text{dim}G\right).
\]
\begin{Rem}
This upper bound, in the Lie group case, has also be deduced using the equivariant moving frame, \cite{FO-1999}. There, the members of the generating set (\ref{genset}) were called \emph{Maurer-Cartan invariants}. 
\end{Rem}

Further assuming that $\H$ acts transitively on $\E$ so that normalizing the $X$'s gives $\mu^i=-\o^i$ and so there are $n$ Maurer-Cartan forms that do not provide any invariants we obtain the upper bound
\beq\label{upperb2}
n\cdot \left( \text{dim}\H_{r-1}-n\right)
\eeq
on the number of generating invariants. Similarly, if $\H$ acts transitively on $J^1(\E)$, normalizing all the first order (and zero order) lifted invariants $\Uh^\alpha_i\equiv e^\alpha_i$ to constants $e^\alpha_i$, the recurrence formula gives
\[
0=d\Uh^\alpha=\Uh^\alpha_i\oi+\text{Maurer-Cartan forms}=e^\alpha_i\oi+\text{Maurer-Cartan forms}
\]
By transitivity we can solve these equations for $m$ Maurer-Cartan forms which then become \emph{constant coefficient} combinations of $\o$'s. We can hence subtract a further $n\cdot m$ from our upper bound (\ref{upperb2}). More generally, using the same argument, if $\H_{r}$ acts transitively on $J^{q^*}(\E)$ we have the upper bound
\[
n\cdot \left(\text{dim}\H_{r-1}-\text{dim}J^{q^*-1}(\E)\right).
\]
Now, $\text{dim}J^{q^*-1}=n+m\binom{q^*-1+n}{n}$ and so the above bound is equal to
\beq\label{upperb3}
n\cdot \left(\text{dim}\H_{r-1}-n-m\binom{q^*-1+n}{n}\right).
\eeq
For a free Lie-group action, we replace $ \text{dim}\H_{r-1}$ with the dimension of the Lie group.
%

\begin{Ex}
Consider the Lie group action
\[
(x, u)\mapsto (X, U)= (\lambda x + a, \lambda u + b),\quad \lambda>0.
\]

The defining equations are
\[
X_u=X_{xx}=U_x=U_{uu}=0,\quad U_u=X_x
\]
and this pseudo-group is quasi horizontal of order $2$. By Theorem \ref{thm:LieTresse}, a second order moving frame will provide us with a generating set of invariants. Normalizing $X=U=0$ the recurrence formula gives
\[
\mu^x=-\o^x=-X_xdx,\quad \mu^u=-\widehat{U}_X\o^x
\]
where $\Uh_X=u_x$ is an invariant. Moving up to the next order, we have
\[
du_x=\Uh_{XX}\o^x=\frac{u_{xx}}{X_x}\o^x
\]
and so (assuming $u_{xx}>0$) we normalize $\Uh_{XX}=1$ to obtain
\[
X_x=u_{xx}
\]
and
\[
0=d\Uh_{XX}=\Uh_{XXX}\o^x-\mu^x_X.
\]
Solving for $\mu^x_X$ we find that
\[
\mu^x_X=\frac{u_{xxx}}{u_{xx}^2}\o^x.
\]
By Theorem \ref{thm:LieTresse}, $u_x$ and $\dis \frac{u_{xxx}}{u_{xx}^2}$ generate the invariant differential algebra. Notice that $u_x$ alone is not a generator, since
\[
du_x=\o^x
\]
and so applying the invariant derivative to $u_x$ generates no new invariants. Hence, we have a minimal set of generators of size 2.

Now, the upper bound (\ref{upperb2}) is
\[
1\cdot\left(3-1\right)=2.
\]
Therefore, the upper bound (\ref{upperb2}) is actually optimal in complete generality.
\end{Ex}
%

Although, as the last example showed, the upper bound (\ref{upperb2}) is optimal it is far from robust; in most cases it will overshoot the actual size of a minimal generating set by quite a lot. 

\begin{Ex}
Consider the Lie-group action
\[
(x, u)\mapsto (x+a, u+p(x))
\]
where $x, a\in\R^n$, $u\in\R^m$, and $p(x)$ is an arbitrary polynomial of total degree at most $d$. This Lie group has dimension 
\[
n+m\binom{n+d}{n}
\]
and acts transitively of $J^d(\E)$. The upper bound (\ref{upperb3}) is 
\beq\label{dmax}
n\cdot\left(n+m\binom{n+d}{n}-n-m\binom{n+d-1}{n}\right)=nm\binom{n+d-1}{n-1}.
\eeq
It is easily seen that a \emph{minimal} set of generating invariants for this Lie group action is $u^\alpha_J$ where $J\in\mathbb{Z}_{\geq0}^n$ with $|J|=d+1$. This minimal generating set has 
\[
m\binom{n+d}{d+1}=m\binom{n+d}{n-1}
\] 
elements. Hence, the upperbound (\ref{dmax}) is just the simplistic upper bound on the number of monomials in $n$ variables of degree $d+1$ obtained by taking the number of monomials of degree $d$ and multiplying by $n$. 
\end{Ex}

\end{section}

\newpage

\section*{Appendix: Eventually free Lie pseudo-groups}

As previously mentioned, the foundational work, \cite{OP-2008, OP-2009}, on the equivariant moving frame for pseudo-groups focused on pseudo-group actions that were \emph{eventually free}, meaning that at some order, the groupoid action, see (\ref{eq:hatw}), of $\H_q$ on some subset $\S_q\subset J^q(\E)$ was free. Equivalently, each group parameter of order $\leq q$ can be normalized in the lifted invariants $\UhaJ$, $|J|\leq q$, restricted to $\S_q$. The key result in \cite{OP-2009}, dubbed \emph{persistence of freeness}, says that if the action is free at order $q$ in $\S_q$, it is also free at order $q+1$ on the pre-image $(\pi^{q+1}_q)^{-1}(\S_q)$ where $\pi^{q+1}_q:J^{q+1}(\E)\to J^q(\E)$ is the canonical projection. In this section we prove that eventually free actions are quasi-horizontal and that persistence of freeness is a trivial consequence thereof. We used persistence of freeness briefly leading up to the Lie-Tresse theorem above.


As a quick warm up for the general case, consider a Lie pseudo-group acting on graphs of maps $x\mapsto u(x)$, $x\in\R$, in $\R^2$. The total derivative on $J^\infty(\R^2)$ is
\[
D_x=\pder{x}+\sum u_{k+1}\pder{u_k},
\]
where $u_k$ is the $k\th$ $x$-derivative of u. The first prolongation of a map $\p(z)=\p(x,u)=(X,U)=Z$ is
\[
\Uh_X=\frac{D_xU}{D_xX}=\frac{U_x+u_xU_u}{X_x+u_xX_u},
\]
but notice that the nominator and denominator are the $\dis\frac{\partial}{\partial y}$ derivatives of $U$ and $X$ in flat coordinates for the graph of $x\mapsto u(x)$ since, on the graph of $u$,
\[
\dis\frac{\partial}{\partial y}=\dis\frac{\partial}{\partial x}+u_x\dis\frac{\partial}{\partial u},
\]
(cf. (\ref{eq:dy})). 

It is easily seen that in general, on the graph of $x\mapsto u(x)$, we have that
\beq\label{eq:dy}
\frac{d^k}{dy^k}Z=D_x^kZ.
\eeq
Defining
\[
D_{X}=\frac{1}{D_xX}D_x,
\]
the higher order prolongations of $\p$ are
\[
\Uh_k=D_{X}U,
\]
which are functions of the $y$-derivatives of $\p$ only. In arbitrary many variables, $n+m$, for $\E=\X\times\U$, the formulas for $\Uh_J$, restricted to a graph of $x\mapsto u(x)$, become rather complicated but they too are functions of the expressions
\[
\frac{\partial^{|K|}}{\partial y^K}Z^a=D^KZ^a,\quad K\in\mathbb{Z}_{\geq0}^n,
\]
where
\[
\frac{\partial}{\partial y^i}=\frac{\partial}{\partial x^i}+u\av_i\frac{\partial}{\partial u\av}.
\]

 
Now assume the action of $\H$ on $\E$ is free at order $q$ in $\V_q\subset J^q(\E)$ and $j^qu\resx \in\V_q$. Let $x\mapsto u(x)$ be a representing map with $q$-jet $j^qu\resx$ at $x$ and such that the entire $q$ jet $j^qu$ lies in $\V_q$. Freeness means that it is possible to solve for every group parameter of order $\leq q$ in a set of equations (restricted to $x\mapsto u(x)$) of the form
\beq\label{eq:ucon}
\Uh^{\alpha_i}_{J^i}=c^{\alpha_i}_{J^i}=\text{constant},
\eeq
for $1\leq i\leq \text{dim}\H_q$. We shall prove that $\H$ is necessarily quasi-horizontal at some order $\leq q$. Writing the determining equations in flat coordinates $(y,w)$ w.r.t $x\mapsto u(x)$, we write $\pmb{Y}^{(q)}$ for those pseudo-group jet coordinates $Z^a_K$, $|K|\leq q$, that are pure $y$-derivatives and $\pmb{W}^{(q)}$ for at most $q\th$ order pseudo-group jet coordinates that are not. We consider two sets of equations; the determining equations for $\H$ of order $\leq q$ and (\ref{eq:ucon}), in flat coordinates for $u$. These can be written
\begin{align*}
F(z,Z, \pmb{Y}^{(q)}, \pmb{W}^{(q)})&=0,\\
G(z, Z, \pmb{Y}^{(q)})&=0,
\end{align*}
since (\ref{eq:ucon}) only involve $\pmb{Y}^{(q)}$. The fact that this system can be uniquely solved for \emph{all} $\pmb{Y}^{(q)}$ and $\pmb{W}^{(q)}$ directly implies that 
\[
F(z, Z, \pmb{Y}^{(q)}, \pmb{W}^{(q)})=0
\]
can be solved for $\pmb{W}^{(q)}$. In particular, \emph{each} pseudo-group jet involving a $w$-derivative can be taken as principal. Thus, eventually free Lie pseudo-group actions are quasi-horizontal and they have horizontal order equal to the order of the pseudo-group. In addition, each pseudo-group jet involving a $u^\alpha$-derivative can be taken as principal.

\begin{theorem}\label{thm:efqh}
Let $\H$ be an eventually free Lie pseudo-group, of order $t^*$, acting on the bundle $\E\to\X$. Then $\H$ is quasi-horizontal of horizontal order $t^*$ and all pseudo-group parameters can be taken to be pure $\X$ derivatives.
\end{theorem}

\begin{Ex}\label{ex:stab}
Consider the Lie group action
\[
(x, u)\mapsto (\lambda x + a, \lambda^{k+1}u + P_k(x)),\quad \lambda>0,
\]
where $P_k(x)$ is an arbitrary polynomial of order $k$. The determining equations for this pseudo group are
\[
X_u=0,~~X_{xx}=0,~~ U_u=X_x^{k+1},~~ \frac{\partial^{k+1}}{\partial x^{k+1}}U=0.
\]
Notice that, since Lie group actions are eventually free almost everywhere, in accordance with the above theorem, \emph{each} pseudo-group jet involving a $u$-derivative can be taken as principal. The order of the pseudo-group is $k+1$.
\end{Ex}

We can now present a rather trivial proof of \emph{persistence of freeness} of Lie pseudo-groups. Let $j^qu\resx\in\V_q\subset J^q(\E)$ lie in a set $\V_q$ on which $\H_q$ acts freely, as above. In flat coordinates w.r.t. a representing map $x\mapsto u(x)$, each pseudo-group jet involving a $w$ derivative can be taken as principal. Let us denote the pseudo-group jet coordinates in $\pmb{Y}^{(q)}$ that have order equal to $q$ by $\pmb{Y}^q$. Then in the determining equations $F^{(q)}(z, Z^{(q)})=0$, ignoring those equations with principal derivatives having a $w$-derivative and those of order less than $q$, we have some subset of $q\th$ order determining equations $\bar{F}^{(q)}(z, \pmb{Y}^{q})=0$ that only involve pure $y$-derivatives of order $q$. Then the $y$-prolongations
\[
\left(\partial_{y_1}\bar{F}^{(q)}\right)(z, \pmb{Y}^{q+1})=0,\ldots, \left(\partial_{y_n}\bar{F}^{(q)}\right)(z, \pmb{Y}^{q+1})=0
\]
give a system $\bar{F}^{(q+1)}(z, \pmb{Y}^{q+1})=0$ of determining equations of order $q+1$ that involve only pseudo-group jets of pure $y$-derivatives.

By freeness, on $\V_q$, there is a collection $\Uh^{\alpha_i}_{J^i}$, $1\leq i\leq (\text{dim}\H_q - \text{dim}\H_{q-1})$, of lifted invariants with $|J^i|=q$, from which all $q\th$ order parametric derivatives can be normalized. That is, the collection $\Uh^{\alpha_i}_{J^i}$, $1\leq i\leq (\text{dim}\H_q - \text{dim}\H_{q-1})$, is full rank in the $q\th$ order parametric derivatives (all of which are pure $y$-derivatives since $\H$ is eventually free). Hence, the set of functions on $\Ht^q$,
\beq\label{eq:fuh}
\left\{\bar{F}^{(q)}(z,  \pmb{Y}^{q}), \Uh^{\alpha_i}_{J^i}\right\}\restrict{1\leq i\leq (\text{dim}\H_q - \text{dim}\H_{q-1})}
\eeq
is full rank in the pseudo-group jets $\pmb{Y}^{q}$ when evaluated at jets in $\V_q$. Now, for any lifted invariant $\UhaJ$, the two collections (once evaluated at $j^qu\resx$)
\beq\label{eq:2col}
\aligned
\{D_{X^i}\UhaJ\}_{1\leq i\leq n},\quad \{\partial_{y_i}\UhaJ\}_{1\leq i\leq n}
\endaligned
\eeq
are invertible linear combinations of each other with coefficients depending on first order jets. Therefore the two collections (\ref{eq:2col}) are equal in rank in the top order pseudo-group jets. Going back to (\ref{eq:fuh}), this means that prolonging all functions in (\ref{eq:fuh}) w.r.t $\partial_{y_i}$ for $1\leq i\leq n$ will give a collection having the same rank in the $\pmb{Y}^{q+1}$ as the system
\beq\label{eq:fuh2}
\left\{\bar{F}^{(q+1)}(z,  \pmb{Y}^{q+1}), \Uh^{\alpha_i}_{J^i, j}\right\}\restrict{1\leq i\leq (\text{dim}\H_q - \text{dim}\H_{q-1}), 1\leq j\leq n}.
\eeq
But since (\ref{eq:fuh}) was full rank in $\pmb{Y}^q$, its prolongation w.r.t. the $y$ variables will be full rank in $\pmb{Y}^{q+1}$. Since (\ref{eq:fuh2}) has the same rank in $\pmb{Y}^{q+1}$ we have proved persistence of freeness.

\newpage






%

\begin{thebibliography}{99}


\bibitem{IMF}
Arnaldsson, {\"O}., Involutive Moving Frames, {\it Differential Geometry and its Applications.}, {\bf 69} (2020).

\bibitem{FO-1999}
Fels, M., Olver, P.J.,  Moving coframes. II. Regularization and theoretical foundations, {\it Acta Appl. Math.} {\bf 55} (1999), 127--208.

\bibitem{KL-2006}
Kruglikov, B., and Lychagin, V., Invariants of pseudogroup actions: Homological methods and finiteness theorem, {\it Int.\ J.\  Geo.\ Meth.\ Mod.\ Phys.} {\bf 3} (2006), 1131--1165.

\bibitem{KL-2016}
Kruglikov, B., and Lychagin, V., Global Lie--Tresse theorem, {\it Selecta Math.} {\bf 22} (2016), 1357--1411.

\bibitem{L-1893}
Lie, S., and Scheffers, G., {\it Vorlesungen über Continuierliche Gruppen mit Geometrischen und Anderen Anwendungen}, B.G. Teubner, Leipzig, 1893.

\bibitem{O-1995}
Olver, P.J., {\it Equivalence, Invariants and Symmetry}, Cambridge University Press, Cambridge, 1995.

\bibitem{OP-2005}
Olver, P.J., and Pohjanpelto, J., Maurer-Cartan forms and the structure of Lie pseudo-groups, {\it Selecta Math.} {\bf 11} (2005), 99--126.

\bibitem{OP-2007}
Olver, P.J., and Pohjanpelto, J., Differential invariants for Lie pseudo-groups, {\it Radon Series Comp.\ Appl.\ Math.} {\bf 1} (2007), 1--27.

\bibitem{OP-2008}
Olver, P.J., and Pohjanpelto, J., Moving frames for Lie pseudo-groups, {\it Canadian J. Math.} {\bf 60} (2008), 1336--1386.

\bibitem{OP-2009}
Olver, P.J., and Pohjanpelto, J., Differential invariant algebras of Lie pseudo-groups, {\it Advances in Mathematics} {\bf 222}(5) (2009), 1746--1792.

\bibitem{S-2010}
Seiler, W., {\it Involution}, Springer, Berling, 2010.

\bibitem{S-2000}
Stormark, O., {\it Lie's structural approach to PDE systems}, Cambridge University Press, Cambridge, 2000.

\bibitem{T-1894}
Tresse, A., Sur les invariants diff{\'e}rentiels des groupes continus de transformations, {\it Acta Math.} {\bf 18} (1894), 1--88.

\end{thebibliography}

\end{document}